[1994/12/01]
\documentclass[draft]{article}
\usepackage{amsmath,amsthm, amssymb}

\chardef\bslash=`\\ 

\newtheorem{theorem}{Theorem}[section]

\theoremstyle{definition}

\theoremstyle{remark}

\newcommand{\T}{\mathcal{T}}  
\newcommand{\curl}{{\rm{}curl}}
\newcommand{\ddiv}{{\rm{}div}}
\newcommand{\grad}{{\rm{}grad}}

\newcommand{\eval}[2][\right]{\relax
  \ifx#1\right\relax \left.\fi#2#1\rvert}

\begin{document}

\title{A mixed FEM for the quad-curl eigenvalue problem
}


\author{Jiguang Sun\\
               Michigan Technological University \\ jiguangs@mtu.edu
}

\maketitle

\begin{abstract}
The quad-curl problem arises in the study of the electromagnetic
interior transmission problem and magnetohydrodynamics (MHD).
In this paper, we study the quad-curl eigenvalue problem and propose a mixed method using edge elements for the computation of the eigenvalues.
To the author's knowledge, it is the first numerical treatment for the quad-curl eigenvalue problem.
Under suitable assumptions on the domain and mesh, we prove the optimal convergence. In addition, we show that the divergence-free condition
can be bypassed. Numerical results are provided to show the viability of the method. 
\end{abstract}

\section{Introduction}
\label{intro}
The quad-curl problem arises in the study of the electromagnetic interior transmission problem in the 
inverse electromagnetic scattering theory for inhomgeneous media
\cite{MonkSun2012SIAMSC} and magnetohydrodynamics (MHD)  equations \cite{ZhengHuXu2011MC}.
The corresponding quad-curl eigenvalue problem has fundamental importance for the analysis and computation of the electromagnetic interior transmission
eigenvalues \cite{Sun2011SIAMNA}.

In general, the numerical computation of the eigenvalue problem starts with the corresponding source problem.
There are only a few results on numerical methods for the quad-curl problem. 
The construction of conforming finite elements with suitable regularity for the quad-curl problem
could be extremely technical and might be prohibitively expensive even if it exists.
Recently Zheng et al. \cite{ZhengHuXu2011MC} propose a non-conforming finite element method to
construct a local element space with a small number of degrees of freedom and impose inter-element continuity along the tangential direction.

In this paper, we first propose a mixed finite element method for the quad-curl problem and  prove the optimal convergence. 
The major advantage of this approach lies in the fact that only curl-conforming edge elements are needed \cite{Nedelec1980NM}. 
Then we employ the mixed method to compute the quad-curl eigenvalues. 
We show the convergence for the eigenvalue problem following the theoretical frame work by Babu\v{s}ka and Osborn \cite{BabuskaOsborn1991}. 
To the author's knowledge, it is the first numerical treatment of the quad-curl eigenvalue problem.
For general theory and analysis on eigenvalue approximations, we refer the readers to 
\cite{BabuskaOsborn1991}, the recent paper \cite{Boffi2010AN} and the references therein.

The rest of the paper is organized as follows. Section 2 contains some preliminaries. In Section 3, we propose a mixed finite element method for the
the quad-curl problem and prove the optimal convergence.
In Section 4, we employ the mixed method for the quad-curl eigenvalue problem and show the convergence rate
of the eigenvalues and the associated generalized eigenspaces. 
In addition, we show that  the divergence-free condition,
which needs to be treated using the Lagrange multiplier for the source problem in general, can be ignored for the eigenvalue problem.
Preliminary numerical results are shown in Section 5. 
Finally, in Section 6, we make conclusions and discuss some future works.

\section{Preliminaries} \label{Preliminary}
\subsection{Function Spaces}
Let $D \subset \mathbb R^3$ be a bounded, simply connected, and convex polyhedral domain. 
The boundary of $D$ is assumed to be connected with unit outward norm $\boldsymbol\nu$.
We denote by $(\cdot, \cdot)$ the $L^2(D)$ inner product and by $\| \cdot \|$ the $L^2(D)$ norm.
The variational approach we shall describe for the quad-curl problem requires
several Hilbert spaces. We define
\[
H^s(\curl{}; D) := \left\{ {\bf u }\in L^2(D)^3\; |\; \curl{}^j\,{\bf u} \in L^2(D)^3, \, 1 \le j \le s \right\}
\]
equipped with the scalar product
\[
({\bf u}, {\bf v})_{H^s(\curl{}; D)} = ({\bf u}, {\bf v})+\sum_{j=1}^s(\curl{}^j\,{\bf u}, \curl{}^j\,{\bf v})
\]
and the corresponding norm $\|\cdot\|_{H^s(\curl{}; D)}$. Next we define spaces with boundary conditions
\begin{eqnarray*}
&&H_0^1(\curl{}; D) := \left\{ {\bf u} \in H^1(\curl{}, D)\; |\; {\bf u} \times  \boldsymbol\nu = 0  \text{ on } \partial D\right\},\\
&&H_0^2(\curl{}; D) := \left\{ {\bf u} \in H^2(\curl{}, D)\; |\; {\bf u} \times  \boldsymbol\nu = 0  \text{ and }  (\curl{} {\bf u}) \times \boldsymbol\nu = 0 \text{ on } \partial D\right\}.
\end{eqnarray*}
Note that $H_0^1(\text{curl}; D)$ in our notation is the standard space $H_0(\text{curl}; D)$.
We also need the space of functions
with square-integrable divergence $H(\ddiv;D)$ defined by
\[
H(\ddiv; D) = \{ {\bf u} \in L^2(D)^3\; |\; \ddiv\, {\bf u} \in L^2(D)\}
\]
equipped with the scalar product
\[
({\bf u}, {\bf v})_{H(\ddiv; D)} = ({\bf u}, {\bf v})+(\ddiv\,{\bf u}, \ddiv\,{\bf v})
\]
and the corresponding norm $\|\cdot\|_{H(\ddiv, D)}$.

Taking the divergence free condition into account, we define
\begin{eqnarray*}
X &=& \left\{ {\bf u} \in H(\curl{}; D) \cap H(\ddiv; D) | \ddiv \, {\bf u} = 0 \text { in } D \right\},\\
Y &=& \left\{ {\bf u} \in H_0(\curl{}; D) \cap H(\ddiv; D) | \ddiv \, {\bf u} = 0 \text { in } D \right\}.
\end{eqnarray*}
For functions in $Y$, the following Friedrichs inequality holds.
\begin{theorem}\label{C351} (see, for example, Corollary 3.51 of \cite{Monk2003})
Suppose that $D$ is a bounded Lipschitz domain. If $D$ is simply connected, and has a connected boundary, there is a
constant $C \ge 0$ such that for every ${\bf u} \in Y$
\begin{equation}\label{FriedrichsCurl}
\|{\bf u}\| \le C \|\curl{}\, {\bf u}\|.
\end{equation}
\end{theorem}

Next we recall a regularity result from  \cite{Girault1988MC}.
\begin{theorem}\label{Girault23} (Theorem 2.3 of \cite{Girault1988MC})
Let $D$ be bounded, convex, simply connected polyhedron with connected boundary $\partial D$ and unit outward normal ${\boldsymbol \nu}$. 
All functions ${\boldsymbol \psi} \in L^2(D)^3$ that satisfy
\[
\ddiv\, {\boldsymbol \psi} = 0, \quad \curl{} \,{\boldsymbol \psi} \in L^2(D)^3,
\quad {\boldsymbol \psi} \cdot {\boldsymbol \nu} = 0\quad ( \text{or } {\boldsymbol \psi} \times {\boldsymbol \nu} = 0) \text{ on } \partial D
\]
belong to $H^1(D)^3$ and
\[
\|{\boldsymbol \psi}\|_{H^1(D)} \le C\|\curl{} \, {\boldsymbol \psi}\|.
\]
\end{theorem}

\subsection{The edge element}
We give a short introduction of the standard edge elements due to N\'ed\'elec \cite{Nedelec1980NM} here. We assume that the domain $D$
is covered by a regular quasi-uniform tetrahedral mesh. We denote the mesh by ${\cal T}_h$ where $h$ is the maximum diameter of the elements in ${\cal T}_h$.
%
Let $P_k$ be the space of polynomials
of maximum total degree $k$ and $\tilde{P}_k$ the space of homogeneous polynomials of degree $k$. We define
\[
R_k = (P_{k-1})^3 \oplus  \{ {\bf p} \in (\tilde{P}_k)^3\; |\; {\bf x} \cdot {\bf p}({\bf x}) = 0\;\text{for all } {\bf x}\in \mathbb{R}^3\}.
\]
The curl-conforming edge element space \cite{Nedelec1980NM} is given by
\[
U_h = \{ {\bf v} \in H(\curl;D)
\;|\; {\bf v}|_K \in R_k \text{ for all } K \in {\cal T}_h\}.
\]
The $H_0(\curl;D)$ conforming edge element space is given by
\begin{equation}
U_{0, h}= \{{\bf u}_h \in U_h\; |\; \, \boldsymbol{\nu} \times {\bf u}_h = 0 \quad \text{ on } \partial D \}
\label{Xhdef}
\end{equation}
which can be easily obtained by taking the degrees of freedom associated with edges or faces on $\partial D$ to vanish.
Let ${\bf r}_h {\bf u} \in U_h$ be the global interpolant \cite{Monk2003}. The following result holds.
\begin{theorem} (Lemma 5.38 of \cite{Monk2003})
Suppose there are constants $\delta > 0$ and $p > 2$ such that ${\bf u} \in H^{1/2+\delta}(K)^3$ and
$\curl \,{\bf u} \in L^p(K)^3$ for each $K \in \tau_h$. Then ${\bf r}_h {\bf u}$ is well-defined and bounded.
\end{theorem}
The following result provides error estimates for the interpolant.
\begin{theorem}\label{Monk2003541} (Theorem 5.41 of \cite{Monk2003})
Let $\tau_h$ be a regular mesh on $D$. Then
\begin{itemize}
\item[(1)] If ${\bf u} \in H^s(D)^3$ and $\curl {\bf u} \in H^s(D)^3$ for $1/2+\delta \le s \le k$ for $\delta >0$ then
\begin{equation}\label{Ierror}
 \|{\bf u} - {\bf r}_h{\bf u}\|_{L^2(D)^3} + \| \curl ({\bf u} - {\bf r}_h{\bf u})\|_{L^2(D)^3} 
  \le Ch^s \left( \|{\bf u}\|_{H^s(D)^3} + \| \curl  {\bf u}\|_{H^s(D)^3}\right).
\end{equation}
\item[(2)] If ${\bf u} \in H^{1/2+\delta}(K)^3$, $0 < \delta \le 1/2$ and $\curl {\bf u}|_K \in R_k$, then
\[
\|{\bf u} - {\bf r}_h{\bf u}\|_{L^2(D)^3}  \le C \left( h_K^{1/2+\delta} \|{\bf u}\|_{H^{1/2+\delta}(K)^3} +h_K \|\curl {\bf u}\|_{L^2(K)^3}\right).
\]
\item[(3)] For $1/2 + \delta \le s \le k$ and $\delta >0$, the following result holds
\[
 \| \curl ({\bf u} - {\bf r}_h{\bf u})\|_{L^2(D)^3}  \le Ch^s \| \curl {\bf u} \|_{H^s(D)^3}.
\]
\end{itemize}
\end{theorem}
The following inverse inequality for edge elements will be useful in the forthcoming error analysis.
\begin{theorem}\label{thesis413} (Lemma 4.1.3 of \cite{thesis2006}).
Let ${\cal T}_h$ be a regular and quasi-uniform mesh for $D$.
Then for ${\bf u}_h \in U_h$  the following holds
\[
\|{\bf u}_h\|_{H(\curl;D)} \le C h^{-1} \|{\bf u}_h\|
\]
for some constant $C$ independent of ${\bf u}_h$ and $h$.
\end{theorem}

\subsection{The curl-curl problem}\label{curl-curl}
The curl-curl problem was studied extensively in literature, for example, see \cite{Kikuchi1989JJAM} and \cite{MonkDemkowicz2000MC}. 
Since it is an important ingredient of the numerical scheme for
 the quad-curl problem, we give a brief description here. Given ${\bf f}\in H(\ddiv; D) $
such that $\ddiv {\bf f}=0$, find ${\bf u}$ such that
\begin{subequations}\label{curl2S}
\begin{align}
\label{curl2SE}&\curl \,\curl\,  {\bf u}= {\bf f} &\text{in } D,\\[1mm]
\label{curl2Sdiv0}&\ddiv \, {\bf u} =0 &\text{in } D,\\[1mm]
\label{curl2SBCa}& {\bf u} \times \boldsymbol\nu = 0 &\text{on } \partial D.
\end{align}
\end{subequations}

To enforce the divergence-free condition on ${\bf u}$, the following mixed formulation is used (see \cite{MonkDemkowicz2000MC}
and references therein). Find $({\bf u}, p) \in H_0(\curl,D) \times H_0^1(D)$ such that
\begin{subequations}\label{curl2Sweak}
\begin{align}
\label{curl2SweakA}(\curl \, {\bf u}, \curl \, \boldsymbol\phi) + ( \grad \, p, {\boldsymbol\phi})&= ({\bf f}, {\bf \boldsymbol\phi}) &\text{for all } {\boldsymbol \phi} \in H_0(\curl, D),\\[1mm]
\label{curl2SweakB}({\bf u}, \grad \, q) &=0 &\text{for all } q \in H_0^1(D).
\end{align}
\end{subequations}
By Theorem 1.1 of \cite{BrezziFortin91Book}, there exists a unique solution $({\bf u}, p)$ of \eqref{curl2Sweak}.
By choosing $ \boldsymbol\phi = \nabla p$ in (\ref{curl2Sweak}) and using the fact that ${\bf f}$ is divergence free, we see that $p=0$.
Furthermore,  ${\bf u}$ satisfies
\[
\|{\bf u}\|_{H(\curl,D)}  \le C\|{\bf f}\|.
\]

Next, we prove a regularity result of the curl-curl problem.
\begin{theorem}\label{curlcurlR}
Let $D$ satisfy the conditions in Theorem \ref{Girault23} and  ${\bf u}$ be the solution of the curl-curl problem \eqref{curl2Sweak}, then we have that
\[
{\bf u} \in H^1(D), \text{ and } \curl \, {\bf u} \in H^1(D).
\]
\end{theorem}
\begin{proof}
Choosing ${\bf v} = {\bf u}$ and $q = p$ in \eqref{curl2Sweak}, we obtain
\[
\|\curl \, {\bf u} \|^2 = ({\bf f}, {\bf u}) \le \|{\bf f}\| \|{\bf u}\|.
\]
The Friedrichs inequality implies
\[
\|{\bf u}\| \le C \| \curl  \,{\bf u}\|.
\]
So $\|\curl  {\bf u}\| \le C  \|{\bf f}\|$. In addition, $\ddiv {\bf u} = 0$, ${\bf u} \times \boldsymbol\nu = 0$ on $\partial D$.
Theorem \ref{Girault23} implies
$
{\bf u} \in H^1(D)^3.
$
Integrating by parts in \eqref{curl2SweakA}, we have that
\[
\curl \,\curl \, {\bf u} = {\bf f} -\grad  \,p \in L^2(D)^3.
\]
Obviously, $\curl \, {\bf u}$ is divergence free, i.e.
\[
\ddiv \, \curl\,  {\bf u} = 0 \quad \text{in } D.
\]
Furthermore, we have
\[
\boldsymbol\nu \cdot \curl \, {\bf u} = 0 \quad \text{on } \partial D.
\]
To see this, note that for all $p \in C^\infty(\overline{D})$, on the one hand,  one has
\[
(\curl\,  {\bf u}, \grad \, p) = ({\bf u}, \curl \, \grad \, p) + \langle \boldsymbol\nu \times {\bf u}, \grad \, p \rangle_{\partial D} = 0.
\]
On the other hand, one has
\[
(\curl \, {\bf u}, \grad \, p) = - (\ddiv\, \curl  \,{\bf u}, p) + \langle \boldsymbol\nu \cdot \curl \, {\bf u}, p  \rangle_{\partial D}.
\]
Then we apply Theorem \ref{Girault23} again to obtain that
 $\curl \, {\bf u} \in H^1(D)^3$.
\end{proof}

Now we describe the edge element method for the curl-curl problem  \cite{MonkDemkowicz2000MC}. 
Let the finite element space for $H_0^1(D)$ be given by
\[
S_h = \left\{ p_h \in H_0^1(D)\; |\; p_h |_K \in P_k \text{ for all } K \in {\cal T}_h \right\}.
\]
It follows that $\grad  S_h \subset U_{0,h}$. A function ${\bf u} \in L^2(D)^3$ is called discrete divergence-free if
\[
({\bf u}, \grad  \xi_h) = 0 \quad \text{ for all } \xi_h \in S_h.
\]
The discrete Helmholtz decomposition (see Section 7.2.1 in \cite{Monk2003}) can be defined via
\[
U_{0, h} = Y_{h} \oplus \grad  S_h
\]
where $Y_{h}$ is given by
\begin{equation}
Y_{h} = \left\{ {\bf u}_h \in U_{0,h} \;|\; ({\bf u}_h, \grad \, \xi_h) = 0 \quad \text{ for all } \xi_h \in S_h\right\}.
\label{X0hdef}
\end{equation}
Then the discrete problem for \eqref{curl2S} can be stated as: Find $({\bf u}_h, p_h) \in U_{0, h} \times S_h$ such that
\begin{subequations}\label{curl2Sdis}
\begin{align}
\label{curl2SdisE}(\curl \, {\bf u}_h, \curl \, \boldsymbol\phi_h) + (\grad \, p_h,  {\boldsymbol \phi}_h)&= ({\bf f}, {\boldsymbol\phi}_h) &\text{for all } {\boldsymbol\phi}_h \in U_{0, h},\\[1mm]
\label{curl2Sdisdiv0}({\bf u}_h, \grad \, q_h) &=0 &\text{for all } q_h \in S_h.
\end{align}
\end{subequations}
Since $\text{grad} S_h \subset U_{0,h}$ we choose $\boldsymbol\phi_h = \text{grad}\, p_h$ in (\ref{curl2SdisE}) and conclude that $p_h = 0$.

\begin{theorem} (Discrete Compactness of $Y_h$, Lemma 7.20 of \cite{Monk2003})\label{fried}
Let $D$ be a bounded simply connected Lipschitz domain and ${\bf u}_h \in Y_h$. 
There exists a positive constant $C$ independent of $h$ such that, for $h $ small enough,
\[
\|{\bf u}_h \| \le C \| \curl  {\bf u}_h \|.
\]
\end{theorem}
\begin{theorem}\label{thm1} (Theorems 2.39 and 2.45 of \cite{Monk2003})
For $h$ small enough, the discrete problem \eqref{curl2Sdis} has a unique solution $({\bf u}_h, p_h) \in U_{0,h} \times S_h$ with $p_h = 0$.
In addition, if $({\bf u}, p) \in H_0(\curl; D) \times H_0^1(D)$ is the solution of \eqref{curl2S} with $p=0$, there exists a constant $C$
independent of $h$, ${\bf u}$, and ${\bf u}_h$ such that
\[
\|{\bf u}-{\bf u}_h\|_{H(\curl;D)}   \le 
C \inf_{{\bf v}_h \in U_{0,h}}\|{\bf u}-{\bf v}_h\|_{H(\curl;D)}.
\]
\end{theorem}

\section{The Quad-curl problem}  \label{quadcurl}
The quad-curl source problem or simply quad-curl problem is defined as follows.
Given a divergence free field ${\bf f}\in L^2(D)^3$, find the vector field ${\bf u}$ such that
\begin{subequations}\label{curl4S}
\begin{align}
\label{curl4SE}&(\curl{})^4\, {\bf u}= {\bf f} &\text{in } D,\\[1mm]
\label{curl4Sdiv0}&\ddiv \,{\bf u} =0 &\text{in } D,\\[1mm]
\label{curl4SBCa}& {\bf u} \times \boldsymbol\nu = 0 &\text{on } \Gamma,\\[1mm]
\label{curl4SBCb}& (\curl \,{\bf u}) \times \boldsymbol\nu = 0 &\text{on } \Gamma.
\end{align}
\end{subequations}

In this section, we first prove the well-posedness of the quad-curl problem in an appropriate functional space.
Then we propose a mixed formulation for it. To this end, we let $V$ and $W$ be given by
\begin{eqnarray}
\label{space4S}
V &:= &\left\{ {\bf u} \in H_0^2(\curl; D)\cap H(\ddiv; D) \;| \; \ddiv\, {\bf u} = 0 \right\},\\
W&:=&\{{\bf u}\in H^2 ({\rm curl};D)\cap H({\rm div};D)\;|\;\ddiv{\bf u}=0\}.
\end{eqnarray}
The curl-curl operator maps $V$ to $L^2(D)^3$.
We define the bilinear form ${\mathcal C}: V\times V \to \mathbb R$
\begin{equation}\label{bilinearV}
\mathcal{C} ({\bf u},{\bf v}) := (\curl \, \curl \, {\bf u}, \curl \, \curl  \,{\bf v}) \quad \text{ for all } {\bf u}, {\bf v} \in V.
\end{equation}
Let ${\bf f} \in H(\ddiv; D) $  such that $\ddiv\, {\bf f}=0$ in $D$. The weak formulation for the quad-curl problem is to find ${\bf u} \in V$ such that
\begin{equation}\label{curlSWeak}
\mathcal{C} ({\bf u}, {\bf v})  = ({\bf f}, {\bf v}) \quad \text{for all } {\bf v} \in V.
\end{equation}

\begin{theorem} 
Let ${\bf f} \in H(\ddiv; D) $ such that $\ddiv {\bf f}=0$ in $D$. There exists a unique solution ${\bf u} \in V$ to \eqref{curlSWeak}.
\end{theorem}
\begin{proof}
Due to the fact that the functions in $V$ are divergence-free,
using the Friedrichs inequality in Theorem~\ref{C351} twice, we see that
the bilinear form $\mathcal C$ is elliptic on $V$. Then the Lax-Milgram lemma implies that there exists a unique solution ${\bf u}$
of \eqref{curlSWeak} in $V$.
\end{proof}

Let $\boldsymbol \phi = \curl \, \curl \,{\bf u}$. We can formally rewrite the quad-curl problem \eqref{curl4S} as a second order system
\begin{subequations}\label{2System}
\begin{align}
\label{2SystemA}\curl^2\, {\boldsymbol \phi}= {\bf f},\\
\label{2SystemB}\curl^2 \, {\bf u} =\boldsymbol\phi,
\end{align}
\end{subequations}
together with the boundary conditions (\ref{curl4SBCa}) and (\ref{curl4SBCb}) for ${\bf u}$.
The mixed formulation for the quad-curl problem can then be stated as follows. Given ${\bf f}\in H(\ddiv; D) $ with $\ddiv \,{\bf f}=0$, find $({\bf u}, \boldsymbol\phi) \in Y \times X$
satisfying
\begin{subequations}\label{2SystemWeak}
\begin{align}
\label{2SystemWeakA}&(\curl  \,{\boldsymbol\phi}, \curl \, {\bf v})= ({\bf f}, {\bf v}), &\text{for all } {\bf v} \in Y,\\
\label{2SystemWeakB}&(\curl \, {\bf u}, \curl\,  {\boldsymbol \psi})  = (\boldsymbol\phi, {\boldsymbol \psi}), &\text{for all } \boldsymbol\psi \in X.
\end{align}
\end{subequations}

In the following, we derive the equivalence of the above mixed formulation to the quad-curl problem. Our argument parallels Section 7.1 of \cite{CiarletFEM} for the biharmonic
equation.
We first notice that the solution of the quad-curl problem is the solution of the following unconstrained minimization problem: Find ${\bf u}$ such that
\begin{equation}\label{CMin}
J({\bf u}) = \inf_{{\bf v} \in V} J({\bf v})
\end{equation}
where
\begin{equation}\label{Jmin}
J({\bf v}) = \frac{1}{2} \int_D |\curl^2 {\bf v}|^2 \,\text{d}x - \int_D {\bf f}\cdot {\bf v}\,\text{d}x
\end{equation}
since (\ref{2SystemWeak}) is the Euler-Lagrange equation for the minimization problem.
Equivalently we consider the constrained minimization problem associated with the quadratic form
\begin{equation}\label{CMinMix}
\mathcal{J}({\bf v}, {\boldsymbol \psi}) = \frac{1}{2} \int_D |\boldsymbol\psi|^2 \,\text{d} x - \int_D {\bf f} \cdot {\bf v}\,\text{d}x
\end{equation}
for $({\bf v}, \boldsymbol\psi) \in V \times L^2(D)^3$ such that $\curl^2 \, {\bf v} = \boldsymbol\psi$.

To derive a variational formulation, we define the space ${\mathcal V}$ as follows
\[
{\mathcal V} := \left\{ ({\bf v}, {\boldsymbol \psi}) \in V \times L^2(D)^3\;|\;
\beta (({\bf v}, {\boldsymbol\psi}), {\boldsymbol \mu}) =0, \text{ for all } {\boldsymbol \mu} \in X \right\},
\]
where
\begin{equation}\label{DefBeta}
\beta (({\bf v}, {\boldsymbol \psi}), {\boldsymbol \mu}) = \int_D \curl  {\bf v}\cdot \curl  {\boldsymbol \mu} \,\text{d}x
 - \int_D {\boldsymbol \psi}\cdot {\boldsymbol \mu}\,\text{d}x.
\end{equation}
Thus the problem can be stated as: Find $({\bf u}, {\boldsymbol\phi}) \in {\mathcal V}$ such that
\[
\int_D {\boldsymbol \phi}\cdot{\boldsymbol \psi}\,\text{d}x =  \int_D {\bf f}\cdot{\bf v}\,\text{d}x \quad \text{for all } ({\bf v}, {\boldsymbol\psi}) \in {\mathcal V}.
\]

\begin{theorem}\label{spaceV}
The mapping
\[
({\bf v}, {\boldsymbol \psi}) \in {\mathcal V} \to \|{\boldsymbol \psi}\|
\]
is a norm over the space $\mathcal V$, which is equivalent to the product norm
\[
({\bf v}, {\boldsymbol \psi}) \in {\mathcal V} \to \left(\|\curl {\bf v}\|^2+\|{\boldsymbol \psi}\|^2\right)^{1/2},
\]
and makes $\mathcal V$ a Hilbert space. In addition, we have
\[
{\mathcal V} := \left\{ ({\bf v}, {\boldsymbol \psi}) \in V \times L^2(D)^3\;|\; \curl^2 {\bf v} = {\boldsymbol \psi} \right\}.
\]
\end{theorem}
\begin{proof} The first part follows directly from the Friedrichs inequality. We give the proof of the second part only.

For ${\bf v} \in H^2(\curl; D)$ and $\boldsymbol\mu \in H(\curl;D)$, we have that
\[
\int_D \curl  {\bf v} \cdot \curl  \boldsymbol\mu \,\text{d}x =
\int_D  \curl  \curl  {\bf v} \cdot {\boldsymbol\mu}\,\text{d}x + \int_{\partial D}  \boldsymbol\nu \times  {\bf v} \cdot \boldsymbol\mu\,\text{d}s.
\]
Let the functions ${\bf v} \in H_0^2(\curl;D)$ and $\boldsymbol\psi \in L^2(D)^3$ be related through
\[
\curl \curl  {\bf v} =\boldsymbol \psi.
\]
For any function $\boldsymbol\mu \in H(\curl;D)$, the above identity
(see Corollary 3.20 of \cite{Monk2003}) shows that
\[
\beta (({\bf v}, \boldsymbol\psi), \boldsymbol\mu) = 0 \quad \text{for all } \boldsymbol\mu \in H(\curl;D).
\]

On the other hand, let ${\bf v} \in Y$ and $\boldsymbol\psi \in L^2(D)^3$ satisfy
$\beta (({\bf v}, \boldsymbol\psi), \boldsymbol\mu) = 0$ for all $ \boldsymbol\mu \in H(\curl; D)$. In particular, we have
\begin{equation}\label{curl2}
\int_D \curl  {\bf v} \cdot \curl  \boldsymbol\mu \,\text{d}x = \int_D \boldsymbol\psi \cdot {\boldsymbol \mu}  \,\text{d}x
\end{equation}
for all $\boldsymbol\mu \in H_0(\curl; D)$. 
Thus $\curl \curl  {\bf v} = {\boldsymbol\psi}$. Using \eqref{curl2} again with $\boldsymbol\mu$ in $H(\curl; D)$,
we obtain $\boldsymbol\nu \times {\bf v}=0$ on $\partial D$.
\end{proof}

\begin{theorem}\label{QCunique}
Let ${\bf u} \in V$ denote the solution of the minimization problem \eqref{CMin}. Then we have
\begin{equation}\label{Nmin}
\mathcal{J}({\bf u}, \curl^2 \,{\bf u}) = \inf_{({\bf v}, \boldsymbol\psi) \in {\mathcal V}} \mathcal{J}({\bf v}, \boldsymbol\psi).
\end{equation}
In addition, the pair $({\bf u}, \curl^2\,{\bf u}) \in \mathcal{V}$ is the unique solution of the minimization problem \eqref{Nmin}.
\end{theorem}
\begin{proof}
The symmetric bilinear form
\[
\left( ({\bf u},\boldsymbol\phi), ({\bf v}, \boldsymbol\psi) \right) \in \mathcal{V} \times \mathcal{V}  \to \int_D \boldsymbol\phi\cdot \boldsymbol\psi \, \text{d}x
\]
is continuous and $\mathcal{V}$-elliptic, and the linear form
\[
({\bf v}, \boldsymbol\psi) \in \mathcal{V} \to \int_D {\bf f}\cdot{\bf v} \, \text{d}x
\]
is continuous. Thus the minimization problem: find $({\bf u}^*, \boldsymbol\phi) \in \mathcal{V}$ such that
\[
\mathcal{J}({\bf u}^*, \boldsymbol\phi) = \inf_{({\bf v}, \boldsymbol\psi) \in \mathcal{V}} \mathcal{J}({\bf v}, \boldsymbol\psi)
\]
has a unique solution, which is also the solution of
\[
\int_D \boldsymbol\phi \cdot \boldsymbol\psi \, \text{d} x = \int_D {\bf f} \cdot {\bf v} \, \text{d}x \quad \text{for all } ({\bf v}, \boldsymbol\psi) \in \mathcal{V}.
\]
From Theorem \ref{spaceV}, we see that ${\bf u}^* \in  V$ and that $\curl^2 {\bf u}^* = \boldsymbol\phi$. We have
\[
\int_D \curl^2 {\bf u} \cdot \curl^2 {\bf v} \, \text{d} x = \int_D {\bf f}\cdot{\bf v} \, \text{d}x,
\]
and thus ${\bf u}^*$ coincides with the solution ${\bf u}$ of \eqref{CMin}.
\end{proof}

Based on the mixed formulation introduced in the previous section, we now present a finite element method for the
quad-curl problem.
Let
\[
{\mathcal V}_h = \left\{ ({\bf v}_h, \boldsymbol\psi_h) \in Y_h \times X_h\;|
\; \beta (({\bf v}_h, \boldsymbol\psi_h), \boldsymbol\mu_h) = 0 \quad \text{for all } \boldsymbol\mu_h \in X_h\right\}
\]
where
\[
X_{h} = \left\{ {\bf u}_h \in U_h \;|\; ({\bf u}_h, \grad \boldsymbol \xi_h) = 0 \quad \text{ for all } \boldsymbol\xi_h \in S_h\right\}.
\]
Note that $ X_{h}  \not \subset X$. Functions in $X_h$ are said to be discrete divergence free.
The discrete problem corresponding to \eqref{CMinMix} is to find $({\bf u}_h, \boldsymbol\phi_h) \in {\mathcal V}_h$ such that
\begin{equation}\label{minmixD}
\mathcal{J}({\bf u}_h, \boldsymbol\phi_h) = \inf_{({\bf v}_h, \boldsymbol\psi_h) \in \mathcal{V}_h} \mathcal{J}({\bf v}_h, \boldsymbol\psi_h)
\end{equation}
\begin{theorem}
The discrete problem \eqref{minmixD} has a unique solution.
\end{theorem}
\begin{proof}
As in the case of the continuous weak formulation, we consider the mapping
\begin{equation}\label{mappingD}
({\bf v}_h, \boldsymbol\psi_h) \in {\mathcal V}_h \to \|\boldsymbol\psi_h\|.
\end{equation}
Suppose $({\bf v}_h, \boldsymbol\psi_h) \in {\mathcal V}_h$ satisfies
\[
\beta (({\bf v}_h, \boldsymbol\psi_h), \boldsymbol\mu_h) = 0 \quad \text{for all } \boldsymbol\mu_h \in X_h.
\]
Choosing $\boldsymbol\mu_h = {\bf v}_h$ we obtain
\begin{equation}\label{curlvh}
\|\curl  {\bf v}_h\|^2 = \int_D \boldsymbol\psi_h \cdot {\bf v}_h \, \text{d}x \le \|\boldsymbol\psi_h\| \|{\bf v}_h\|.
\end{equation}
Applying the discrete Friedrichs inequality, we have
\[
\|\curl  {\bf v}_h\|^2 \le C \|\boldsymbol\psi_h\|\, \|\curl  {\bf v}_h\|.
\]
Therefore,
\[
\left( \|\curl   {\bf v}_h\|^2+\|\boldsymbol\psi_h\|^2 \right)^{1/2} \le (C^2+1)^{1/2} \|\boldsymbol\psi_h\|.
\]
Hence the mapping \eqref{mappingD} is a norm over ${\mathcal V}_h$.
The rest of the proof follows by an argument similar to that used for the proof of Theorem \ref{QCunique}.
\end{proof}
As a consequence of the above theorem, the element $({\bf u}_h, \boldsymbol\phi_h) \in {\mathcal V}_h$ satisfies
\begin{equation}\label{uhphih}
\int_D {\bf \boldsymbol\phi}_h \cdot  {\bf \boldsymbol\psi}_h \, \text{d}x = \int_D {\bf f} \cdot {\bf v}_h \,\text{d}x
\quad \text{for all } ({\bf v}_h, \boldsymbol\psi_h) \in {\mathcal V}_h.
\end{equation}

\begin{theorem}
Let $({\bf u}, \boldsymbol\phi)$ and $({\bf u}_h, \boldsymbol\phi_h)$ be the solutions of \eqref{Nmin} and \eqref{minmixD}, respectively, and
assume that ${\bf u}\in H^3(\curl; D)$.
There exists a constant $C$ independent of $h$ such that
\begin{eqnarray} \label{ErrEstA}
&&\|\curl   {\bf u} - \curl   {\bf u}_h\| + \|\curl^2 {\bf u} - \boldsymbol\phi_h\|  \le \\
 \nonumber&& C \left( \inf_{({\bf v}_h, \boldsymbol\psi_h) \in {\mathcal V}_h}
\left( \|\curl  {\bf u} - \curl  {\bf v}_h \|+
\| \curl^2 {\bf u} - \boldsymbol\psi_h \|\right)
+ \inf_{\boldsymbol\mu_h \in X_h} \| \curl^2 \,{\bf u} - \boldsymbol\mu_h \|_{H(\curl)}\right).\nonumber
\end{eqnarray}
\end{theorem}
\begin{proof}
Using the assumption that ${\bf u} \in H^3(\curl; D)$,  we have
\[
\int_D  \curl  ( \curl^2  {\bf u})\cdot \curl  {\bf v} \, \text{d}x
 = \int_D \curl^2  {\bf u} \cdot \curl^2  {\bf v}\, \text{d}x
  = \int_D {\bf f}\cdot {\bf v} \, \text{d}x,
\]
for all ${\bf v} \in {\mathcal D}(D)^3$, the space of smooth functions with compact support in $D$. Hence for all ${\bf v} \in H_0(\curl;D)$, the following
holds 
\begin{equation}\label{allv}
\int_D  \curl  ( \curl^2  {\bf u})\cdot \curl  {\bf v} \, \text{d}x
 =\int_D {\bf f} \cdot {\bf v} \, \text{d}x.
\end{equation}
Thus for any ${\bf v} \in H_0(\curl;D)$ and $\boldsymbol\psi \in L^2(D)^3$, we have
\[
\beta\left(({\bf v}, \boldsymbol\psi), \curl^2 {\bf u} \right)
= \int_D {\bf f}\cdot {\bf v} \, \text{d}x - \int_D \boldsymbol\psi \cdot \curl^2 {\bf u} \, \text{d}x.
\]
For any $({\bf v}_h, \boldsymbol\psi_h) \in {\mathcal V}_h$ and $\boldsymbol\mu_h \in X_h$, using the fact that
$\beta( ({\bf v}_h, \boldsymbol\psi_h), \boldsymbol\mu_h ) = 0$, \eqref{allv}, and \eqref{uhphih},
we have
\begin{eqnarray}
&& \beta\left( ({\bf u}_h - {\bf v}_h, \boldsymbol\phi_h - \boldsymbol\psi_h), \curl^2 {\bf u} - \boldsymbol\mu_h\right)\label{buh}\\
&=& \int_D \curl  ({\bf u}_h - {\bf v}_h) \cdot \curl  (\curl^2 {\bf u} - \boldsymbol\mu_h) \, \text{d}x\nonumber\\
&& \qquad \qquad    - \int_D (\boldsymbol\phi_h - \boldsymbol\psi_h) \cdot (\curl^2 {\bf u} - \boldsymbol\mu_h) \, \text{d}x \nonumber\\
& =&  \int_D \curl  ({\bf u}_h - {\bf v}_h) \cdot \curl  \curl^2 {\bf u}  \, \text{d}x
-\int_D \curl  ({\bf u}_h - {\bf v}_h) \cdot \curl  \boldsymbol\mu_h \, \text{d}x \nonumber\\
&& \qquad \qquad - \int_D (\boldsymbol\phi_h-\boldsymbol\psi_h) \cdot \curl^2 {\bf u}   \, \text{d}x + \int_D  (\boldsymbol\phi_h - \boldsymbol\psi_h) \cdot \boldsymbol\mu_h  \, \text{d}x  \nonumber\\
&=&  \int_D \curl  ({\bf u}_h - {\bf v}_h) \cdot \curl  \curl^2 {\bf u}  \, \text{d}x  - \int_D (\boldsymbol\phi_h-\boldsymbol\psi_h) \cdot \curl^2 {\bf u}   \, \text{d}x \nonumber\\
&=& \int_D {\bf f} \cdot ({\bf u}_h - {\bf v}_h) \, \text{d}x  - \int_D (\boldsymbol\phi_h-\boldsymbol\psi_h) \cdot \curl^2 {\bf u}   \, \text{d}x \nonumber\\
&=& \int_D \boldsymbol\phi_h \cdot (\boldsymbol\phi_h - \boldsymbol\psi_h)  \, \text{d}x  - \int_D (\boldsymbol\phi_h-\boldsymbol\psi_h) \cdot \curl^2 {\bf u}   \, \text{d}x\nonumber \\
&=& - \int_D (\curl^2{\bf u} - \boldsymbol\phi_h)\cdot(\boldsymbol\phi_h - \boldsymbol\psi_h) \, \text{d}x. \nonumber
\end{eqnarray}
On the other hand, for all $\boldsymbol\mu_h \in X_h$, one has
\begin{eqnarray*}
&& \int_D \curl  {\bf u}_h \cdot \curl  \boldsymbol\mu_h \text{d}\,x = \int_D \boldsymbol\phi_h \cdot \boldsymbol\mu_h \text{d}x, \\
&& \int_D \curl  {\bf v}_h \cdot \curl  \boldsymbol\mu_h \text{d}\,x = \int_D \boldsymbol\psi_h \cdot \boldsymbol\mu_h \text{d}x.
\end{eqnarray*}
Taking the difference and letting $\boldsymbol\mu_h = {\bf u}_h - {\bf v}_h$, we have
\[
\int_D \curl  ({\bf u}_h - {\bf v}_h) \cdot \curl  ({\bf u}_h - {\bf v}_h) \text{d}x = \int_D (\boldsymbol\phi_h - \boldsymbol\psi_h) \cdot  ({\bf u}_h - {\bf v}_h)  \text{d}x
\]
which implies
\begin{equation}\label{uhvh}
\|  \curl  ({\bf u}_h - {\bf v}_h) \| \le C \| \boldsymbol\phi_h -\boldsymbol \psi_h\|
\end{equation}
where $C$ is the constant in the discrete Friedrichs inequality.

Using the above inequality and the definition of $\beta(\cdot,\cdot)$ together with \eqref{buh}, we have
\begin{eqnarray*}
&&\left| \int_D (\curl^2{\bf u} - \boldsymbol\phi_h)\cdot(\boldsymbol\phi_h - \boldsymbol\psi_h) \, \text{d}x \right| \\
&=& \left| \beta\left( ({\bf u}_h - {\bf v}_h, \boldsymbol\phi_h - \boldsymbol\psi_h), \curl^2 {\bf u} - \boldsymbol\mu_h\right) \right| \\
&\le& \| \curl  ({\bf u}_h - {\bf v}_h)\| \, \| \curl  (\curl^2 {\bf u} - \boldsymbol\mu_h)\|
+\| \boldsymbol\phi_h - \boldsymbol\psi_h \| \, \| \curl^2 {\bf u} - \boldsymbol\mu_h \| \\
&\le& C \| \boldsymbol\phi_h - \boldsymbol\psi_h \| \, \| \curl  (\curl^2 {\bf u} - \boldsymbol\mu_h)\| +
\| \boldsymbol\phi_h - \boldsymbol\psi_h \| \, \| \curl^2 {\bf u} - \boldsymbol\mu_h \| \\
&\le& C_1 \| \boldsymbol\phi_h - \boldsymbol\psi_h \| \, \| \curl^2 {\bf u} - \boldsymbol\mu_h \|_{H(\curl)}
\end{eqnarray*}
where $C_1 = \max\{C, 1\}$.
Thus we have that
\begin{eqnarray*}
&&\| \boldsymbol\phi_h - \boldsymbol\psi_h \|^2 \\
&=& -\int_D (\boldsymbol\phi_h - \boldsymbol\psi_h)\cdot(\curl^2 {\bf u} - \boldsymbol\phi_h)\,\text{d}x
+ \int_D (\boldsymbol\phi_h - \boldsymbol\psi_h)\cdot(\curl^2 {\bf u} - \boldsymbol\psi_h)\,\text{d}x \\
&\le& C_1 \| \boldsymbol\phi_h - \boldsymbol\psi_h \| \, \| \curl^2 {\bf u} - \boldsymbol\mu_h \|_{H(\curl)}
 + \| \boldsymbol\phi_h - \boldsymbol\psi_h \| \,\| \curl^2 {\bf u} - \boldsymbol\psi_h \|
\end{eqnarray*}
and hence
\[
\| \boldsymbol\phi_h - \boldsymbol\psi_h \| \le C_1\| \curl^2 {\bf u} - \boldsymbol\mu_h \|_{H(\curl)} +
\| \curl^2 {\bf u} - \boldsymbol\psi_h \|.
\]
Moreover, we have that
\begin{eqnarray*}
&&\| \curl  {\bf u} - \curl  {\bf u}_h \| + \| \curl^2 {\bf u} - \boldsymbol\phi_h \| \\
&\le& \| \curl  {\bf u} - \curl  {\bf v}_h \| + \| \curl  {\bf v}_h - \curl  {\bf u}_h \|
 + \| \curl^2 {\bf u} - \boldsymbol\psi_h \| + \| \boldsymbol\psi_h - \boldsymbol\phi_h\| \\
&\le& \| \curl  {\bf u} - \curl  {\bf v}_h \| + \| \curl^2 {\bf u} - \boldsymbol\psi_h \|
 +(1+C) \| \boldsymbol\psi_h - \boldsymbol\phi_h\|
\end{eqnarray*}
where we have used \eqref{uhvh}.
Combining the above inequalities, we obtain
\begin{eqnarray*}
&& \| \curl  {\bf u} - \curl  {\bf u}_h \| + \| \curl^2 {\bf u} - \boldsymbol\phi_h \| \\
&\le& \|\curl  {\bf u} - \curl  {\bf v}_h \| + \| \curl^2 {\bf u} - \boldsymbol\psi_h \| \\
&& \qquad \qquad     +(1+C) \left( C_1\| \curl^2 {\bf u} - \boldsymbol\mu_h \|_{H(\curl)} +
\| \curl^2 {\bf u} - \boldsymbol\psi_h) \| \right) \\
&\le& \|\curl  {\bf u} - \curl  {\bf v}_h \| + (2+C) \| \curl^2 {\bf u} - \boldsymbol\psi_h \| 
 +(1+C)C_1 \| \curl^2\, {\bf u} - \boldsymbol\mu_h \|_{H(\curl)}
\end{eqnarray*}
which completes the proof by taking the infimum of the right side of the above inequality over all $({\bf v}_h,\boldsymbol{\psi}_h)\in {\cal V}_h$ 
and ${\boldsymbol\mu}_h\in X_h$.
\end{proof}

\begin{theorem}\label{ErrorEstC}
Let $({\bf u}, \boldsymbol\phi)$ and $({\bf u}_h, \boldsymbol\phi_h)$ be the solutions of \eqref{Nmin} and \eqref{minmixD}, respectively.
Let $\alpha(h)=C_1/h$ where $C_1$ is the constant in Theorem \ref{thesis413}.
Then there exists a constant $C$ independent of the mesh size $h$ such that
\begin{eqnarray*}
&&\|\curl  {\bf u} - \curl  {\bf u}_h\| + \|\curl^2 \, {\bf u} - \boldsymbol\phi_h\| \\
&& \qquad \le C \left\{ \left( 1+\alpha(h) \right) \inf_{{\bf v}_h \in X_h} \|\curl  {\bf u} - \curl  {\bf v}_h\|
+ \inf_{\boldsymbol\mu_h \in X_h} \|\curl^2 \,{\bf u} + \boldsymbol\mu_h\|_{H(\curl)} \right\}.
\end{eqnarray*}
\end{theorem}
\begin{proof}
Let $({\bf v}_h, \boldsymbol\psi_h) \in {\mathcal V}_h$ and $\boldsymbol\mu_h \in X_h$ be arbitrary. Letting ${\bf w}_h = \boldsymbol\mu_h+\boldsymbol\psi_h$, we have that
$\beta(({\bf v}_h, \boldsymbol\psi_h), {\bf w}_h) = 0$, i.e.,
\[
\int_D \curl  {\bf v}_h \cdot \curl  {\bf w}_h \,\text{d} x - \int_D \boldsymbol\psi_h \cdot {\bf w}_h \, \text{d}x = 0.
\]
Using the fact that $\boldsymbol\nu \times \curl  {\bf u} = 0$ on $\partial D$, we obtain
\[
\int_D \curl \curl  {\bf u} \cdot {\bf w}_h \, \text{d} x = \int_D \curl  {\bf u} \cdot \curl  {\bf w}_h \, \text{d}x.
\]
Combining the above two equations, we obtain
\[
\int_D (\curl \curl  {\bf u} - \boldsymbol\psi_h) \cdot {\bf w}_h \, \text{d} x = \int_D \curl  ({\bf u} - {\bf v}_h) \cdot \curl  {\bf w}_h \, \text{d} x.
\]
Hence we have that
\begin{eqnarray*}
\left | \int_D (\curl \curl  {\bf u} - \boldsymbol\psi_h) \cdot {\bf w}_h \, \text{d} x \right |
&\le& \| \curl   {\bf u} - \curl  {\bf v}_h \|\, \| \curl  {\bf w}_h\| \\
&\le& \alpha(h) \| \curl   {\bf u} - \curl  {\bf v}_h \|\, \| {\bf w}_h\|.
\end{eqnarray*}
Since  ${\bf w}_h = \boldsymbol\mu_h+\boldsymbol\psi_h$, we have
\begin{eqnarray*}
\| {\bf w}_h\|^2 &=& \int_D (\boldsymbol\mu_h + \curl \curl  {\bf u}) \cdot {\bf w}_h \, \text{d} x + \int_D (  \boldsymbol\psi_h - \curl \curl  {\bf u} ) \cdot {\bf w}_h \, \text{d} x \\
&\le& \|\boldsymbol\mu_h + \curl \curl  {\bf u}\| \, \| {\bf w}_h\| + \alpha(h) \| \curl   {\bf u} - \curl  {\bf v}_h \|\, \| {\bf w}_h\|.
\end{eqnarray*}
From this inequality, we have that
\begin{eqnarray*}
\| \curl \curl  {\bf u} -\boldsymbol \psi_h \| &\le& \| \curl \curl  {\bf u} + \boldsymbol\mu_h \| + \| {\bf w}_h\| \\
&\le& 2 \| \curl \curl  {\bf u} + \boldsymbol\mu_h \| +\alpha(h)  \| \curl   {\bf u} - \curl  {\bf v}_h \|,
\end{eqnarray*}
and thus,
\begin{eqnarray*}
&&\inf_{({\bf v}_h, \boldsymbol\psi_h) \in {\mathcal V}_h} \left( \| \curl   {\bf u} - \curl  {\bf v}_h \| + \| \curl \curl  {\bf u} - \boldsymbol\psi_h \| \right)  \\
&& \qquad \le (1+\alpha(h))  \inf_{{\bf v}_h \in X_h}\| \curl   {\bf u} - \curl  {\bf v}_h \| + 2 \inf_{\boldsymbol\mu_h \in X_h} \| \curl \curl  {\bf u} + \boldsymbol\mu_h \|.
\end{eqnarray*}
Combination of this inequality and \eqref{ErrEstA} completes the proof.
\end{proof}

\begin{theorem}\label{thm}
Let $({\bf u}, \boldsymbol\phi)$ and $({\bf u}_h, \boldsymbol\phi_h)$ be the solutions of \eqref{Nmin} and \eqref{minmixD}, respectively. 
Furthermore, we assume that $\curl^i {\bf u} \in H^s(D)^3, i = 1,2,3$ 
and $s$ is the same as in Theorem \ref{Monk2003541}. 
Then there exists a constant $C$ independent of
the mesh size $h$ such that
\begin{equation}\label{ErrorOrder}
\|\curl  {\bf u} - \curl  {\bf u}_h\|  +  \|\curl^2 \, {\bf u} - \boldsymbol\phi_h\| \le C h^{s-1} \| \curl \, {\bf u}\|_{H^s(D)^3}.
\end{equation}
\end{theorem}
\begin{proof}
From Theorem \ref{Monk2003541}, we have that
\begin{eqnarray*}
&&\inf_{{\bf v}_h \in U_{0,h}} \| \curl  {\bf u} - \curl  {\bf v}_h \| \le C h^s \| \curl \, {\bf u}\|_{H^s(D)^3}, \\
&&\inf_{{\boldsymbol \mu}_h \in U_h} \|\curl^2 {\bf u} - \boldsymbol\mu_h\|_{H(\curl)} \le C h^s \left( \|\curl^2 \,{\bf u}\|_{H^s(D)^3}+\|\curl^3\, {\bf u} \|_{H^s(D)^3}\right),
\end{eqnarray*}
for some constants C independent of $h$. 
Using Theorem \ref{ErrorEstC}, we obtain that
\begin{eqnarray*}
&&\|\curl  {\bf u} - \curl  {\bf u}_h\| + \|\curl^2 \, {\bf u} - \boldsymbol\phi_h\| \\
&\le&  C\left( 1+\frac{C_1}{h} \right)  h^s \| \curl \, {\bf u}\|_{H^s(D)^3} + C h^s \left( \|\curl^2 \,{\bf u}\|_{H^s(D)^3}+\|\curl^3\, {\bf u} \|_{H^s(D)^3}\right) \\
&\le&  C h^{s-1} \| \curl \, {\bf u}\|_{H^s(D)^3} + C h^s \left( \|\curl^2 \,{\bf u}\|_{H^s(D)^3}+\|\curl^3\, {\bf u} \|_{H^s(D)^3}\right) \\
&\le& C h^{s-1} \| \curl \, {\bf u}\|_{H^s(D)^3},
\end{eqnarray*}
if $h$ is small enough.
\end{proof}
We now use the theory from Section \ref{curl-curl} to prove an $L^2$ norm convergence result for ${\bf u}-{\bf u}_h$.  
Of course, since we are using the edge element of \cite{Nedelec1980NM} (the so-called edge elements of the first kind), 
the convergence rate in $L^2$ can not be better than the convergence rate in $H(\text{curl})$. So 
nothing would be gained from a duality argument.
\begin{theorem}\label{L2conv}
Under the conditions of Theorem \ref{thm}, there exists a constant $C$ independent of ${\bf u}$, ${\bf u}_h$ and $h$ such that
\[
\Vert {\bf u}-{\bf u}_h\Vert \leq C \left(h^{s}\Vert {\bf u}\Vert_{H^s(D)^3}+ h^{s-1}\Vert \curl {\bf u}\Vert_{H^{s}(D)^3}\right).
\]
\end{theorem}

\begin{proof}
Let ${\bf v}_h\in Y_h$ be the first component of the solution of (\ref{curl2Sdis}) with ${\bf f}=\curl\curl{\bf u}$ so that ${\bf u}$ is the exact solution. 
By Theorem \ref{thm1} and Theorem \ref{Monk2003541}, we have that
\begin{equation}
\Vert {\bf u}-{\bf v}_h\Vert_{H(\curl;D)}\leq Ch^{s-1} \left(\Vert {\bf u}\Vert_{H^s(D)^3}+\Vert \curl {\bf u}\Vert_{H^{s}(D)^3}\right).\label{uvcurl}
\end{equation}
Then, using the triangle inequality and the discrete Friedrichs inequality in Theorem \ref{fried}, we have that
\begin{eqnarray*}
\Vert {\bf u}-{\bf u}_h\Vert&\leq &\Vert {\bf u}-{\bf v}_h\Vert+\Vert{\bf v}_h-{\bf u}_h\Vert
\\&\leq& \Vert {\bf u}-{\bf v}_h\Vert+C\Vert\curl({\bf v}_h-{\bf u}_h)\Vert\\
&\leq &
 C(\Vert {\bf u}-{\bf v}_h\Vert_{H(\curl;D)}+\Vert\curl({\bf u}-{\bf u}_h)\Vert).
\end{eqnarray*}
Combination of Theorem \ref{thm} and (\ref{uvcurl})  completes the result.
\end{proof}

\section{The Quad-curl Eigenvalue Problem}
The quad-curl eigenvalue problem is to find $\lambda$ and the 
non-trivial vector field function ${\bf u}$ such that 
\begin{subequations}\label{curl4SEig}
\begin{align}
\label{curl4SEigE}&(\curl{})^4\, {\bf u}= \lambda {\bf u} &\text{in } D,\\[1mm]
\label{curl4SEigdiv0}&\ddiv {\bf u} =0 &\text{in } D,\\[1mm]
\label{curl4SEigBCa}& {\bf u} \times \boldsymbol\nu = 0 &\text{on } \partial D,\\[1mm]
\label{curl4SEigBCb}& (\curl{} {\bf u}) \times \boldsymbol\nu = 0 &\text{on } \partial D.
\end{align}
\end{subequations}
We call $\lambda$ a quad-curl eigenvalue and ${\bf u}$ the associated eigenfunction.
Due to the well-posedness of the quad-curl problem, we can define an operator $T: L^2(D)^3 \to  L^2(D)^3$ such that $T{\bf f} = {\bf u}$ for \eqref{curlSWeak}. 
It is obvious that $T$ is self-adjoint. 
Furthermore, because of the compact imbedding of $V$ into $L^2(D)^3$, $T$ is a compact operator.

The quad-curl eigenvalue problem is to find $(\lambda, {\bf u}) \in \mathbb R \times V$ such that 
\begin{equation}\label{curlEigSWeak}
\mathcal{C} ({\bf u}, {\bf q})  = \lambda({\bf u}, {\bf q}) \quad \text{for all } {\bf q} \in V.
\end{equation}
It is clear that $\lambda$ is an eigenvalue satisfying \eqref{curlEigSWeak} if and only if $\mu = 1/\lambda$ is an eigenvalue of $T$.

\begin{theorem} There is an infinite discrete set of quad-curl eigenvalues $\lambda_j > 0, j = 1, 2, \ldots$ and corresponding eigenfunctions ${\bf u}_j \in V$, ${\bf u}_j \ne {\bf 0}$ such that
\eqref{curlEigSWeak} is satisfied and $0 < \lambda_1 \le \lambda_2 \le \ldots$. Furthermore
\[
\lim_{j \to \infty} \lambda_j = \infty.
\]
The eigenfucntions $({\bf u}_j, {\bf u}_l)_{L^2(D)^3} = 0$ if $j \ne l$.
\end{theorem}
\begin{proof}
Applying the Hilbert-Schmidt theory (see, for example, Theorem 2.36 of \cite{Monk2003}), we immediately have the above theorem.
\end{proof}

\begin{theorem} The quad-curl eigenvalues coincide with the non-zero eigenvalues of the following problem.
Find $(\lambda, {\bf u}) \in \mathbb R \times H_0^2(\text{curl}; D)$ such that 
\begin{equation}\label{curlEigSLWeak}
\mathcal{C} ({\bf u}, {\bf q})  = \lambda({\bf u}, {\bf q}) \quad \text{for all } {\bf q} \in H_0^2(\text{curl}, D).
\end{equation}
\end{theorem}
\begin{proof}
Applying the Helmholtz decomposition (see, for example, Lemma 4.5 of \cite{Monk2003}), we write
\[
{\bf u} = {\bf u}_0 + \nabla p, \quad {\bf u}_0 \in V, \, p \in S
\]
where 
$S = H_0^1(D)$. 
Note that we can write $V$ in a different way as
\[
V= \{ {\bf w} \in H_0(\text{curl}; D)| ({\bf w}, \nabla \xi) = 0 \text{ for all } \xi \in S \}.
\]
In \eqref{curlEigSLWeak}, taking ${\bf q} = \nabla \xi$, we have that 
\[
\lambda \left (({\bf u}_0+\nabla p), \nabla \xi \right) = \left(\text{curl curl} \,({\bf u}_0 + \nabla p), \text{curl curl} \,\nabla \xi \right) =0
\]
since $\text{curl}\, \nabla \xi = 0$. In addition, the Helmholtz decomposition implies $({\bf u}_0, \nabla \xi) = 0$. Thus we have that
\[
\lambda(\nabla p, \nabla \xi) = 0 \quad \text{ for all } \xi \in S.
\]

If $\lambda \ne 0$, letting $\xi = p$ we have $\nabla p = 0$, and thus $p=0$ due to the boundary condition on $p$.
Thus we have that ${\bf u} ={\bf u}_0 \in V$ and ${\bf u} \in V$ since $ {\bf u} \in  H_0^2(\text{curl}; D)$.

If $\lambda = 0$, we have that
\[
 \left(\text{curl curl} \,({\bf u}_0 + \nabla p), \text{curl curl}  \,{\bf q} \right) = 0  \text{ for all } {\bf q} \in H_0^2(\text{curl}; D).
\]
Since $\text{curl}\, \nabla p = 0$, we obtain 
\begin{equation}
(\text{curl curl} {\bf u}_0, \text{curl curl} \,{\bf q}) = 0 \text{ for all } {\bf q} \in V.
\end{equation}
Choosing ${\bf q} = {\bf u}_0$, we have ${\bf u}_0 = {\bf 0}$ due to the Friedrichs inequality (see, for example, Corollary 4.8 of \cite{Monk2003}). 
Thus $\lambda = 0$ is an eigenvalue of infinite multiplicity  of \eqref{curlEigSLWeak} and the corresponding
eigenfunctions are ${\bf u} = \nabla p$, for $p \in S$. In summary, non-zero eigenvalues of \eqref{curlEigSLWeak}
are exactly quad-curl eigenvalues of \eqref{curlEigSWeak}. 
\end{proof}

Now we introduce a mixed formulation for the quad-curl problem.
We define
\[
a({\bf u}, {\bf v}) = ({\bf u}, {\bf v}), \quad b({\bf u}, {\bf v}) = -(\text{curl} \, {\bf u}, \text{curl} \, \bf{v}).
\]
%
We define two solution operators $A: L^2(D)^3 \to L^2(D)^3$ and $B: L^2(D)^3 \to L^2(D)^3$ given by
\begin{equation}
A{\bf f} = {\bf w}, \quad B{\bf f} = {\bf u}
\end{equation}
for a divergence-free field ${\bf f}$ in $L^2(D)^3$. Hence we can rewrite the mixed method for the quad-curl problem as 
\begin{subequations}\label{4curlSO}
\begin{align}
\label{4curlSA}&a(A{\bf f}, {\bf v}) + {b({\bf v}, B{\bf f})}=0, \quad &\text{for all } {\bf v} \in X, \\
\label{4curlSB}&b(A{\bf f}, {\bf q}) = - ({\bf f}, {\bf q}),  \quad &\text{for all } {\bf q} \in Y.
\end{align}
\end{subequations}

Then the quad-curl eigenvalue problem in mixed form can be written as: find $\lambda \in {\mathbb R}$, 
$0 \ne ({\bf u}, {\bf w}) \in W \times X$ satisfying
\begin{subequations}\label{MixQCEig}
\begin{align}
\label{MixQCEigA} a({\bf w}, {\bf v}) + {b(\bf{v}, {\bf u})}=0, \quad \text{for all } {\bf v} \in X, \\
\label{MixQCEigB} b({\bf w}, {\bf q}) = -\lambda ({\bf u}, {\bf q}),  \quad \text{for all } {\bf q} \in Y.
\end{align}
\end{subequations}
It is easy to see that if $(\lambda, ({\bf u}, {\bf w}))$ is an eigenpair of \eqref{MixQCEig}, then $\lambda B {\bf u} = {\bf u}, {\bf u} \ne {\bf 0}$, i.e.,
$(\lambda, {\bf u})$ is a quad-curl eigenpair. 
If  $\lambda B {\bf u} = {\bf u}, {\bf u} \ne {\bf 0}$, then there exists a ${\bf w} \in X$ such that $(\lambda, ({\bf u}, {\bf w}))$ is an eigenpair of \eqref{MixQCEig}.

Next we describe a mixed method for the quad-curl eigenvalue problem based on the method for the quad-curl problem in the previous section.
We assume that the domain $D$
is covered by a regular and quasi-uniform tetrahedral mesh. 
We denote the mesh by ${\T}_h$ where $h$ is the maximum diameter of the elements in ${\T}_h$.
Note that the discrete compactness holds for $Y_{h}$ holds \cite{Monk2003},
i.e., there exists a positive constant $C$ independent of $h$ such that if
${\bf u}_h \in Y_{h}$, for $h $ small enough,
\[
\|{\bf u}_h \|_{L^2(D)^3} \le C \| \curl  {\bf u}_h \|_{L^2(D)^3}.
\]
An immediate consequence is that $b(\cdot, \cdot)$ is coercive on $Y_{h} \times Y_{h}$.
We define
\[
X_{h}= \left\{ {\bf u}_h \in U_h \;|\; ({\bf u}_h, \grad \xi_h) = 0 \quad \text{ for all } \xi_h \in S_h\right\}.
\]
The mixed finite element method for the quad-curl problem can be stated as follows.
For a divergence-free ${\bf f} \in L^2(D)^3$, find $A_h{\bf f} \in Y_{h}$, $B_h{\bf f} \in X_{h}$ such that
\begin{subequations}\label{4curlSD}
\begin{align}
\label{4curlSDA}&a(A_h {\bf f}, {\bf v}_h) + {b({\bf v}_h, B_h{\bf f})}=0, \quad &\text{for all } {\bf v}_h \in X_{h}, \\
\label{4curlSDB}&b(A_h {\bf f}, {\bf q}_h) = -({\bf f}, {\bf q}_h),  \quad &\text{for all } {\bf q}_h \in Y_{h}.
\end{align}
\end{subequations}

From Theorems \ref{thm} and \ref{L2conv} in the previous section, we have that
\begin{eqnarray}
\label{BBHF0}\|(B-B_h) {\bf f} \| &\le& Ch^{s-1} \|B{\bf f}\|_{s}, \\
\label{AAhf0}\|(A-A_h){\bf f}\|&\le& C h ^{s-1} \|\text{curl} B{\bf f}\|_{s}.
\end{eqnarray}
In the following, we assume that the regularities $\|{\bf u}\|_s \le C\|{\bf f}\|$ and $\|\text{curl} {\bf u}\|_{s} \le C\|{\bf f}\|$ holds for
the quad-curl problem for some constant $C$.
Note that when $s=2$, the above regularity results is
a consequence of Theorem \ref{C351} and the fact that ${\bf u}$ is the solution
of the quad-curl problem.
Thus we have the norm convergence
\[
\lim_{h \to 0} \|B-B_h\| = 0 \quad \text{and} \quad
\lim_{h \to 0} \|A-A_h\| = 0.
\]

Employing the above mixed method for the quad-curl eigenvalue problem, we obtain the following discrete problem.
Find $\lambda_h \in \mathbb R$, $({\bf u}_h, {\bf w}_h) \in Y_{h}\times X_{h}$ such that
\begin{subequations}\label{4curlSDEig}
\begin{align}
\label{4curlSDA}&a({\bf w}_h, {\bf v}_h) + {b({\bf v}_h, {\bf u}_h)}=0, \quad &\text{for all } {\bf v}_h \in X_{h}, \\
\label{4curlSDB}&b({\bf w}_h, {\bf q}_h) = -\lambda_h ({\bf u}_h, {\bf q}_h),  \quad &\text{for all } {\bf q}_h \in Y_{h}.
\end{align}
\end{subequations}

Similar to the continuous case,  
we show that when we compute the discrete eigenvalues, we can bypass the divergence-free condition, namely,
we can work with $U_{0,h}$ and $U_h$ instead of the spaces $Y_{h}$ and $X_{h}$ and only keep non-zero
eigenvalues.

\begin{theorem} The discrete quad-curl eigenvalues of \eqref{4curlSDEig} coincide with the non-zero 
eigenvalues of the following problem.
Find $\lambda_h \in \mathbb R$ and ${\bf u}_h \in U_{0,h}, {\bf w}_h \in U_h$ such that
\begin{subequations}\label{4curlSDEigB}
\begin{align}
\label{4curlSDBBA}&({\bf w}_h, {\bf v}_h) - (\text{curl} \,{\bf v}_h, \text{curl} \,{\bf u}_h)=0, \quad &\text{for all } {\bf v}_h \in U_{h}, \\
\label{4curlSDBBB}&(\text{curl}\, {\bf w}_h, \text{curl}\, {\bf q}_h) = -\lambda_h ({\bf u}_h, {\bf q}_h),  \quad &\text{for all } {\bf q}_h \in U_{0,h}.
\end{align}
\end{subequations}
\end{theorem}
\begin{proof}
We define a discrete subspace $S_h$ of $S$ as
\[
S_h = \{ p_h \in H_0^1(D) | \, p_h|_K \in P_k \quad \text{ for all } K \in \mathcal{T}_h \}
\]
and write 
\[
{\bf u}_h = {\bf u}_h^0 + \nabla \varphi_h, \quad {\bf u}_h^0 \in Y_{h}, \varphi_h \in S_h.
\]
Letting ${\bf q}_h = \nabla \xi_h$ in \eqref{4curlSDBBB}, we have that
\[
0=({\bf w}_h, \text{curl} \, \nabla \xi_h)= -\lambda_h ({\bf u}_h, \nabla \xi_h) = \lambda_h (\nabla \varphi_h, \nabla \xi_h)  \quad \text{for all } \xi_h \in S_h.
\]
Thus either $\lambda_h = 0$ or $(\nabla \varphi_h, \nabla \xi_h)=0$  for all $\xi_h \in S_h$.
It is clear that if $\lambda_h \ne 0$, we have $(\nabla \varphi_h, \nabla \xi_h) =0$ for all $\xi_h \in S_h$ which
implies $\nabla \varphi_h = 0$. Thus ${\bf u}_h = {\bf u}_h^0 $ which is discrete divergence-free.
\end{proof}

Now we are ready to prove the convergence  for the mixed method. 
Let $\mu$ be a non-zero eigenvalue of $B$.
The ascent $\alpha$ of $\mu -B$ is defined as the smallest integer such that $N((\mu-B)^\alpha) = N((\mu-B)^{\alpha+1})$,
where $N$ denotes the null space. Let $m=\text{dim}N((\mu-B)^\alpha)$ be the algebraic multiplicity of $\mu$. 
The geometric multiplicity of $\mu$ is $\text{dim}N(\mu-B)$. Note that since $B$ is self-adjoint, the two
multiplicities are same.
Then there are $m$ eigenvalues of $B_h$, $\mu_1(h), \ldots, \mu_m(h)$ such that
\begin{equation}
\lim_{h \to 0} \mu_j(h) = \mu, \quad \text{for } j = 1, \ldots, m.
\end{equation}
\begin{theorem}

Let $\lambda =1/\mu$ be an exact quad-curl eigenvalues with multiplicity $m$ and $\lambda_{j, h}, j = 1, \ldots, m$ be a
computed eigenvalue using the mixed finite element method. Then we have that
\begin{equation}
|\lambda - {\lambda}_{j,h}| \le Ch^{2s-2}
\end{equation}
for some constant $C$. 
\end{theorem}
\begin{proof} For \eqref{BBHF0}, and \eqref{AAhf0}, we have that
\[
\|(B-B_h) \| \le Ch^{s-1} \quad \text{and} \quad
\|(A-A_h)\| \le C h ^{s-1}.
\]
Then the theorem is verified using Theorem 11.1 of \cite{BabuskaOsborn1991}.
\end{proof}

\section{Numerical Examples}
In this section, we show some preliminary results. 
As we see previously, we can ignore the divergence free condition when we compute the quad-curl eigenvalues.
This enable us to work with the edge element space directly. Namely, we only need to solve the following problem.
Find $\lambda \in \mathbb R$, $({\bf u}_h, {\bf w}_h) \in U_{0, h}\times U_{h}$ such that
\begin{subequations}\label{4curlSDEigX}
\begin{align}
\label{4curlSDXA}&a({\bf w}_h, {\bf v}_h) + {b({\bf v}_h, {\bf u}_h)}=0, \quad &\text{for all } {\bf v}_h \in U_{h}, \\
\label{4curlSDXB}&b({\bf w}_h, {\bf q}_h) = -\lambda_h ({\bf u}_h, {\bf q}_h),  \quad &\text{for all } {\bf q}_h \in U_{0, h}.
\end{align}
\end{subequations}
Let $\{ {\boldsymbol{\phi}}_i, i=1, \ldots, N \}$ be a basis for $U_{0,h}$ and $\{ {\boldsymbol{\phi}}_i, i=1, \ldots, N, N+1, \ldots, M \}$ be a basis for $U_h$.
Then the matrix form corresponding to the above equations is given by
\begin{equation}\label{AEP}
\begin{pmatrix} {\bf 0}_{N\times N}&\mathcal{C}_{N \times M}\\-\mathcal{C}_{M\times N}&\mathcal{M}_{M\times M}\end{pmatrix}
=\lambda \begin{pmatrix} \mathcal{M}_{N\times N}&{\bf 0 }_{N \times M}\\{\bf 0}_{M\times N}&{\bf 0}_{M\times M}\end{pmatrix}
\end{equation}
where
\begin{eqnarray*}
\mathcal{C}_{N \times M}(i,j) &=& ({\rm{curl}}{\boldsymbol{\phi}}_j, {\rm{curl}}{\boldsymbol{\phi}}_i), i=1, \ldots N, j = 1, \ldots, M,  \\
\mathcal{C}_{N \times M}(i,j) &=& ({\rm{curl}}{\boldsymbol{\phi}}_j, {\rm{curl}}{\boldsymbol{\phi}}_i), i=1, \ldots M, j = 1, \ldots, N,\\
\mathcal{M}_{N \times N}(i,j)&=&( {\boldsymbol{\phi}}_j, {\boldsymbol{\phi}}_i), i=1, \ldots N, j = 1, \ldots, N,\\
\mathcal{M}_{M \times M}(i,j)&=&( {\boldsymbol{\phi}}_j, {\boldsymbol{\phi}}_i), i=1, \ldots M, j = 1, \ldots, M.
\end{eqnarray*}
The resulting algebraic eigenvalue problem is solved by Matlab '{\it eigs}' on a desktop computer.

We consider two domains: the unit ball and the unit cube. 
Due to the restriction of the computational power available, the largest matrices we can compute 
are obtained using a rather coarse mesh ($h \approx 0.1$). This is why we are not able to show
the convergence order. However, the eigenvalues seem to converge for all examples. Of course,
a better eigenvalue solver on a more powerful machine is very much desired. 

We show the results on a few meshes on both domains in Tables \ref{Ball} and \ref{Cube}. The degrees of
freedom is denoted by DoF in the table. For the unit ball, three meshes, corresponding to the mesh sizes
$h \approx 0.3$, $h \approx 0.2$, $h \approx 0.15$, are used. For both the linear $p=1$ and quadratic $p=2$ edge elements,
we see some numerical evidence of the convergence of the proposed method Table \ref{Ball}.


\begin{center}
\begin{table}
\caption{The first quad-curl eigenvalues for the unit ball on a few meshes using the linear and quadratic edge element. 
Besides the computed eigenvalue, we also show the degrees of freedom (DoF) of the discrete problems which equals the
dimension of the matrices defined in \eqref{AEP}. }
\label{Ball}
\begin{center}
\begin{tabular}{|c|c|c|c|c|}
\hline
	{}&{$h \approx 0.3$}&{$h \approx 0.2$}&$h \approx 0.15$&$h \approx 0.1$\\
\hline
	{$p=1$}&{201.6299 (6580) }&{199.3129 (21282)}&197.8822 (49792)& 196.6903(917576)\\
\hline
	{$p=2$}&{206.0821 (35608)}&{200.7491(115348)}&198.7244(270072)&-\\
\hline
\end{tabular}
\end{center}
\end{table}
\end{center}

For the unit cube, we use three meshes corresponding to the mesh sizes
$h \approx 0.4$, $h \approx 0.2$, $h \approx 0.1$ and obtain numerical evidence of the 
convergence of the proposed method in Table \ref{Cube} for both the linear $p=1$ and quadratic $p=2$ edge elements.

\begin{center}
\begin{table}
\caption{The first quad-curl eigenvalues for the unit cube on a few meshes using the linear and quadratic edge element. 
Besides the computed eigenvalue, we also show the degrees of freedom (DoF) of the discrete problems which equals the
dimension of the matrices defined in \eqref{AEP}.}
\label{Cube}
\begin{center}
\begin{tabular}{|c|c|c|c|c|}
\hline
	{}&{$h \approx 0.4$}&{$h \approx 0.2$}&$h \approx 0.1$ & $h \approx 0.05$ \\
\hline
	{$p=1$}&{1.5209e3 (734) }&{1.6210e3 (5121)}&1.6922e3 (40359)&1.7072e3(325911)\\
\hline
	{$p=2$}&{1.8240e3 (3924)}& 1.7486e3(27698)&1.7236e3 (218854)&-\\

\hline
\end{tabular}
\end{center}
\end{table}
\end{center}

\section{Conclusions and future works}
In this paper, we study the quad-curl eigenvalue problem and propose a mixed finite element method. To the author's
knowledge, it is the first numerical treatment for the problem. 

The theory in this paper is validate for edge elements of order two or higher. However,
numerical results indicate that the proposed mixed method also converges for the linear edge element. 
Hence it is interesting to prove the convergence for the edge element of order one in future.

Since we are dealing three dimensional examples, even the mesh size of the finest mesh is rather large. All the numerical examples
are done using Matlab on a desktop with 8G memory.  We plan to investigate the numerical convergence when more powerful machines with
larger memory and/or efficient softwares are available.

\section*{Acknowledgements}
The research was supported in part by NSF grant DMS-1016092.
The author would like to thank Prof. Peter Monk for helpful discussions.



\end{document}